\documentclass[12pt,reqno,bookmarksnumbered,raiselinks,breaklinks,colorlinks,plainpages]{amsart}

\usepackage{amsmath, amsthm, amscd, amsfonts, amssymb, graphicx, color}
\usepackage{url}
\usepackage{hyperref}

\renewcommand{\epsilon}{\varepsilon}

\newcommand\ubd{\overline{\mbox{\rm dim}}_{\rm B}\,} 
\newcommand\lbd{\underline{\mbox{\rm dim}}_{\rm B}\,} 
\newcommand\da{\underline{\mbox{\rm dim}}_{\,\theta}} 

\newcommand\uda{\overline{\mbox{\rm dim}}_{\,\theta}} 
\newcommand{\dimh}{\dim_{\rm H}}

\newtheorem{theo}{Theorem}[section]
\newtheorem{cor}[theo]{Corollary}
\newtheorem{lem}[theo]{Lemma}
\newtheorem{prop}[theo]{Proposition}

\newtheorem{defi}[theo]{Definition}
\newtheorem{rem}[theo]{Remark}

\newtheorem{rmk}[theo]{Remark}
\newtheorem{question}{Question}

\title[Intermediate Dimensions of Complementary Sets]{Attainable forms of intermediate dimensions of complementary sets}

\author{Nicolas Angelini}
\address{Department of Mathematics, Universit\`a degli Studi di Trento. Trento, Italy.}
\address{Departamento de Matem\'atica, FCFMyN, Universidad Nacional de San Luis, Argentina.}
\email{nicolas.angelini.2015@gmail.com}


\begin{document}

\begin{abstract}
Given a positive, non-increasing sequence $a$ with sum $1$, we consider the
family $\mathcal{C}_a$ of closed subsets of $[0,1]$ whose complementary open
intervals have lengths given by the terms of $a$. We study which functions
$\theta\mapsto\overline{\dim}_{\,\theta}E$ arise as the upper intermediate
dimensions of sets $E\in\mathcal{C}_a$. Under a natural condition on the
decay of $a$, we show that these are exactly the functions in the class
introduced by Banaji and Rutar that take the value $\overline{\dim}_{\rm B}C_a$
at $\theta=1$ and are bounded above by the upper intermediate dimensions of
the decreasing Cantor set $C_a$. We also show that the lower bound given by
the countable set $D_a$ is rigid: if it is attained at one value of $\theta$,
then it is attained at all larger values.
\end{abstract}
\maketitle

\tableofcontents

\section{Introduction}
The $\theta$-intermediate dimensions were introduced by Falconer, Fraser and
Kempton~\cite{FFK} as a continuum of dimensions that interpolates between the
Hausdorff dimension and the box-counting dimensions. For $\theta\in[0,1]$, the
lower and upper $\theta$-intermediate dimensions $\da$ and $\uda$ are defined
by restricting the diameters of the covering sets to an interval of the form
$[\delta^{1/\theta},\delta]$ (see Definition~\ref{adef}). Thus $\theta=0$
corresponds to the Hausdorff dimension and $\theta=1$ to the lower and upper
box dimensions. For every bounded set $E\subset\mathbb{R}^d$,
\[
\dimh E\leq \da E\leq \uda E\leq \ubd E
\qquad\text{and}\qquad
\da E\leq \lbd E ,
\]
and both functions of $\theta$ are non-decreasing and continuous on $(0,1]$.
These dimensions have been studied in many settings, for instance for
fractional Brownian images~\cite{Brownian1,Brownian2}, orthogonal projections~\cite{BFF,angelini}, elliptical polynomial spirals~\cite{BFF2}. There are also many generalizations, for instance \cite{Banaji} for intermediate dimensions in metric spaces and \cite{angelini2} for intermediate dimensions of measures.

A natural problem is to determine which functions arise as intermediate
dimensions. This problem was solved by Banaji and Rutar~\cite{Banaji-R22}. For
$0\leq\lambda\leq\alpha\leq d$, they introduced the class
$\mathcal{H}(\lambda,\alpha)$ of non-decreasing functions
$h:[0,1]\to[\lambda,\alpha]$ which are continuous on $(0,1]$ and whose growth
is controlled by the Dini condition
\[
D^+h(\theta)\leq\frac{(h(\theta)-\lambda)(\alpha-h(\theta))}{(\alpha-\lambda)\,\theta},
\qquad \theta\in(0,1).
\]
They proved the following. As $F$ ranges over the bounded subsets of
$\mathbb{R}^d$ with lower Assouad dimension $\lambda$ and Assouad dimension
$\alpha$, the pairs $(\da F,\uda F)$ are exactly the pairs $(h_1,h_2)$ of
functions in $\mathcal{H}(\lambda,\alpha)$ with $h_1\leq h_2$ and
$h_1(0)=h_2(0)$ (see Theorem~\ref{BanajiRutar}). In this paper we study the
analogous problem within a classical family of subsets of the line: the
complementary sets of a sequence.

Let $a=\{a_n\}$ be a positive, non-increasing sequence with $\sum_n a_n=1$.
The class $\mathcal{C}_a$ of \emph{complementary sets} of $a$ consists of all
closed sets $E\subset[0,1]$ whose complement in $[0,1]$ is a disjoint union of
open intervals whose lengths are exactly the terms of $a$. This class contains
sets of very different nature. Two of them play a central role. The first is
the countable set $D_a$, a decreasing sequence of points whose consecutive gaps
are $a_1,a_2,\dots$. The second is the decreasing Cantor set $C_a$, obtained by
removing the gaps in decreasing order along a binary Cantor construction (see
Section~\ref{Comp sets}). The question of which values a given dimension takes
on $\mathcal{C}_a$ has a long history. For the Hausdorff dimension it goes back
to Besicovitch and Taylor~\cite{BeyTa}, and the set of attainable values is
the closed interval
$[0,\dimh C_a]$~\cite{BeyTa,cabrelli2010classifying,garcia2007dimension};
related questions for packing measures were considered in~\cite{hare2013sizes}.
In contrast, the lower and upper box dimensions do not depend on the
arrangement of the gaps, so each of them takes a single value on
$\mathcal{C}_a$~\cite[Chapter 3.2]{falconer1997techniques}. For the Assouad and lower
Assouad dimensions, García, Hare and Mendivil~\cite{garcia2018assouad} showed
that $C_a$ and $D_a$ are extremal. They also showed that, under suitable decay
conditions on $a$, the attainable values form the intervals $[\dim_A C_a,1]$
and $[0,\dim_L C_a]$, respectively.

The intermediate dimensions of complementary sets were studied
in~\cite{angelini2026intermediate} for each fixed $\theta$. There it was shown
that $D_a$ and $C_a$ are again extremal: for every $E\in\mathcal{C}_a$,
\[
\uda D_a\leq\uda E\leq\uda C_a ,
\]
where the upper bound requires a mild assumption on $a$, and the analogous
inequalities hold for the lower intermediate dimensions. Moreover,
$\uda D_a=\theta B/(1-(1-\theta)B)$ with $B=\ubd C_a$. This is the smallest
value allowed by the bounds of Banaji and Rutar for a set with upper box
dimension $B$. Finally, it was proved that every value between these two
bounds is attained for each fixed $\theta$, both for the upper and for the
lower intermediate dimensions, provided that
\[
2<\inf_n\frac{s_n(a)}{s_{n+1}(a)}\leq\sup_n\frac{s_n(a)}{s_{n+1}(a)}<\infty .
\]
Here $s_n(a)=2^{-n}\sum_{i\geq 2^n}a_i$ is the average length of the intervals
of the $n$-th step in the construction of $C_a$. This condition, which we call
hypothesis~\eqref{hyp H} (see Section~\ref{sec attainable functions}),
combines the conditions under which the Assouad and the lower Assouad
dimensions of complementary sets fill an interval~\cite{garcia2018assouad}.

These results describe, for each fixed $\theta$, the set of attainable values.
They do not describe how the intermediate dimensions of a single
complementary set vary with $\theta$, which is the information captured by
the theorem of Banaji and Rutar. This leads to the question, raised
in~\cite[Question 1]{angelini2026intermediate}, of describing the families
\[
\mathcal{F}^U(a):=\{\theta\mapsto\uda E : E\in\mathcal{C}_a\}\]
and
\[\mathcal{F}^L(a):=\{\theta\mapsto\da E : E\in\mathcal{C}_a\}.
\]
Three necessary conditions for a function $h$ to belong to $\mathcal{F}^U(a)$
follow from the results above. By the theorem of Banaji and Rutar,
$h\in\mathcal{H}(0,1)$. Since all sets in $\mathcal{C}_a$ have the same upper
box dimension, $h(1)=\ubd C_a$. Finally, by the extremality of $C_a$,
$h\leq\uda C_a$.

Our main result shows that, under hypothesis~\eqref{hyp H}, these conditions
are also sufficient. More precisely (Theorem~\ref{main functions} and
Corollary~\ref{characterization}), if $a$ satisfies~\eqref{hyp H}, then
$\mathcal{F}^U(a)$ is exactly the set of functions $h\in\mathcal{H}(0,1)$ such
that
\[
h(1)=\ubd C_a \qquad\text{and}\qquad h(\theta)\leq\uda C_a\ \text{ for all }\theta\in[0,1],
\]
with the usual convention $\overline{\dim}_{\,0}=\dimh$. In particular,
$\mathcal{F}^U(a)$ depends on $a$ only through the function
$\theta\mapsto\uda C_a$, and it is much larger than the family of functions
realized in~\cite{angelini2026intermediate}. For instance, suppose that the
box dimension of $C_a$ exists. Since $\dimh C_a=\lbd C_a$, we then have
$\uda C_a=\dim_{\rm B} C_a$ for every $\theta$, and the second condition is
automatic. In this case $\mathcal{F}^U(a)$ consists of all functions
$h\in\mathcal{H}(0,1)$ with $h(1)=\dim_{\rm B}C_a$. By the theorem of Banaji
and Rutar, these are exactly the upper intermediate dimension functions of
bounded subsets of $\mathbb{R}$ with this upper box dimension. This applies,
for example, to the middle third Cantor set.

We also obtain a sharp lower bound with a rigidity property. Every
$h\in\mathcal{F}^U(a)$ satisfies $h\geq\uda D_a$. If $h$ coincides with
$\uda D_a$ at a single $\theta_0\in(0,1]$, then it coincides with it on the
whole interval $[\theta_0,1]$ (Proposition~\ref{lower bound rigidity}).
Moreover, equality at some $\theta_0\in(0,1)$ forces the lower Assouad and
Assouad dimensions of the set to be $0$ and $1$, respectively
(Proposition~\ref{extremal assouad}). These facts follow from a simple
reformulation of the Dini condition defining $\mathcal{H}(0,1)$ as the
monotonicity of the quotient $h(\theta)/(\theta(1-h(\theta)))$
(Lemma~\ref{monotone quotient}). This reformulation also gives a short proof
of the lower bound $\uda D_a\leq\uda E$.

The proof of Theorem~\ref{main functions} follows the strategy of Banaji and
Rutar~\cite{Banaji-R22}. One first constructs, by concatenating
\emph{mountains} and \emph{valleys}, a function whose infima over windows of
length $\log(1/\theta)$ reproduce $h(\theta)$. This function is then
converted into a sequence of scales for a Cantor-type set. In our setting,
two new difficulties appear. First, the set must be built using only the terms
of $a$ as gaps. This forces its scales to be, up to constants, at most those
of $C_a$, while the hypothesis $h\leq\uda C_a$ only provides information on
infima over windows. We overcome this by combining the scales produced by the
mountains with those of $C_a$, and by placing each mountain on a window where
the covering numbers of $C_a$ are large. Second, the function $h$ may saturate
the Dini condition defining $\mathcal{H}(0,1)$. In that case, any set whose
upper intermediate dimensions are given by $h$ must have lower Assouad
dimension $0$ and Assouad dimension $1$. The ratios between consecutive
scales of such a set cannot remain bounded away from $0$ and from $1/2$. This
cannot be achieved with the decreasing Cantor sets used
in~\cite{angelini2026intermediate} (see Remark~\ref{rmk why schemes}). For
this reason we work with limit sets of homogeneous binary Cantor schemes. The
gaps of these schemes are terms of $a$, but they are not arranged in
decreasing order, and their intermediate dimensions depend only on their
scales (Lemma~\ref{dimension scheme}). The remaining terms of $a$ are used to
build a countable set, which does not increase the dimension because of the
lower bound above.

The corresponding problem for the lower intermediate dimensions remains open.
The main obstruction is that these dimensions are not finitely stable. As a
consequence, the gluing of the Cantor part and the countable part used in our
construction does not control the lower intermediate dimensions of the union.

The paper is organized as follows. In Section~2 we recall the definitions and
basic properties of the dimensions involved. In Section~\ref{Comp sets} we
recall the main facts about complementary sets, including the results
of~\cite{angelini2026intermediate} that we use.
Section~\ref{sec attainable functions} contains our results. In
Subsection~\ref{4.1} we recall the results of Banaji and Rutar that we need.
In Subsection~\ref{4.2} we prove the lower bound and its rigidity. In
Subsections~\ref{4.3} and~\ref{4.4} we develop the homogeneous Cantor schemes
and the mountains. Finally, in Subsection~\ref{4.5} we prove the main theorem
and the characterization of $\mathcal{F}^U(a)$.

\section{Preliminaries}

\subsection{Definitions}

Given a non-empty set $A \subset \mathbb{R}$, we write $|A|$ for the diameter of the set.\\ We write \(a_n \asymp b_n\) to mean that there exist constants \(c_1, c_2 > 0\), independent of \(n\), such that 
\(c_1 b_n \leq a_n \leq c_2 b_n\) for all \(n\).

The Hausdorff dimension of $A$ will be denoted by $\dim_H A$, and $\lbd A$ and $\ubd A$ will denote the lower and upper box-counting dimensions of $A$, respectively. The reader can refer to \cite{FBook} for more information on these dimensions.

This document concerns the $\theta$-intermediate dimensions, which were introduced in \cite{FFK} and are defined as follows.

\begin{defi}\label{adef}
Let $F \subseteq \mathbb{R}$ be bounded. For $0 \leq \theta \leq 1$, we define the \emph{lower $\theta$-intermediate dimensions} of $F$ by
\begin{align*}
 \da F =  \inf \big\{& s \geq 0 : \ \forall\ \epsilon >0, \ \text{and all }\ \delta_0>0, \text{there}\ \exists\ 0<\delta\leq \delta_0, \ \text{and } \\
 & \{U_i\}_{i \in I}\  : F \subseteq \cup_{i \in I} U_i \ : \delta^{1/\theta} \leq  |U_i| \leq \delta \ \text{and }\ \sum_{i \in I} |U_i|^s \leq \epsilon  \big\}.
\end{align*}

Similarly, we define the \emph{upper $\theta$-intermediate dimensions} of $F$ by
\begin{align*}
\uda F =  \inf \big\{& s \geq 0 : \ \forall\ \epsilon >0, \ \text{there}\ \exists\ \delta_0>0, \ : \ \forall\ 0<\delta\leq \delta_0, \ \text{there}\ \exists \\
 & \{U_i\}_{i \in I}\  : F \subseteq \cup_{i \in I} U_i \ : \delta^{1/\theta} \leq  |U_i| \leq \delta \ \text{and }\ \sum_{i \in I} |U_i|^s \leq \epsilon  \big\}.
\end{align*}
\end{defi}

For all $\theta \in (0,1]$ and a compact set $A \subset \mathbb{R}$, we have $\dim_H A \leq \uda A \leq \ubd A$ and similarly with the lower $\theta$ dimension, replacing $\ubd$ with $\lbd$. This spectrum of dimensions has the property of being continuous for $\theta \in (0,1]$, and sometimes it is also continuous at $\theta=0$. The reader can find more information in \cite{FFK}.

\begin{rmk}
    It is straightforward to note that the value of $\da F$ and $\uda F$ remains unchanged if the covering condition $\delta^{1/\theta} \le |U_i| \le \delta$ is replaced by $c_1\delta^{1/\theta} \le |U_i| \le c_2\delta$ for any fixed constants $c_1, c_2 > 0$.
\end{rmk}

We also need to recall the definitions of the Assouad dimension and the lower Assouad dimension. Given a non-empty set $E \subset \mathbb{R}$, $N_r(E)$ will denote the minimum number of closed balls of radius $r$ needed to cover $E$.

The \emph{Assouad dimension} of $E$ is defined as 
\begin{align*}
    \dim_A E= \inf\{ \alpha : & \text{ there exist constants } c,\rho>0 \text{ such that }\\
    & \sup_{x \in E} N_r (B(x,R) \cap E) \leq c\left(\frac{R}{r}\right)^\alpha \text{ for all } 0<r<R<\rho \}.
\end{align*}
In the same way, the \emph{lower Assouad dimension} of $E$ is defined as
\begin{align*}
    \dim_L E= \sup\{ \alpha : & \text{ there exist constants } c,\rho>0 \text{ such that }\\
    & \inf_{x \in E} N_r (B(x,R) \cap E) \geq c\left(\frac{R}{r}\right)^\alpha \text{ for all } 0<r<R<\rho \}.
\end{align*}

The following inequalities hold for compact sets:
$$\dim_L E \leq \dim_H E \leq \da E\leq \uda E\leq \ubd E\leq \dim_A E.$$

A key inequality involving these dimensions, which we will use later, is given in the following result. 
\begin{prop} [Corollary 2.8 \cite{Banaji-R22}] \label{banajis bound}
  
If $\dim_L E<\dim_A E$, then for all $\theta \in (0,1]$,
\[
     \uda E \geq \frac{\theta\dim_A E (\ubd E - \dim_L E) + \dim_L E (\dim_A E - \ubd E)}{\theta(\ubd E - \dim_L E) + (\dim_A E-\ubd E)}.
\]
The same holds replacing $\uda$ and $\ubd$ with $\da$ and $\lbd$ everywhere.
\end{prop}

\section{Complementary sets and their fractal dimensions}\label{Comp sets}
Let $a=\{a_n\}$ be a positive, non-increasing, summable sequence with sum $1$. Let $I_a=[0,1]$. $\mathcal{C}_a$ will denote the family of all closed sets $E$ contained in $I_a$ that are of the form $E=I_a \setminus \displaystyle\cup_{j \geq 1} U_j$, where $\{U_j\}$ is a disjoint family of open intervals contained in $I_a$ such that $|U_j|=a_j$ for all $j$. The sets in $\mathcal{C}_a$ are called the complementary sets of $a$. Note that if $\mathbf{a}=\{\mathbf{a}_n\}$ is any rearrangement of the sequence $\{a_n\}$, then $\mathcal{C}_a=\mathcal{C}_{\mathbf{a}}$.

An important set in $\mathcal{C}_a$ is the \emph{countable} set, defined as follows:
$$D_a=\left\{\sum_{i=k}^{\infty} a_i\right\}_{k\in\mathbb{N}}.$$ 
As the reader notes, $D_a$ is a countable, decreasing set of points with decreasing gaps given by the elements of the sequence $a$.

Another set of great importance in this family is the non-increasing Cantor set associated with the sequence $a$, which is constructed inductively in the following way. In the first step, remove from $I_a$ an open interval of length $a_1$, leaving two closed intervals $I_{1,1}$ and $I_{1,2}$. Having performed the $k$-th step, we have $2^k$ closed intervals $I_{k,1},...,I_{k,2^{k}}$ contained in $I_a$. Then, from each interval $I_{k,j}$, remove an open interval of length $a_{2^k +j-1}$ to obtain two closed intervals $I_{k+1,2j -1}$ and $I_{k+1,2j}$. Finally, define the Cantor set as $$C_a := \displaystyle\bigcap_{k=1}^{\infty}\bigcup_{j=1}^{2^{k}} I_{k,j}.$$

Given a non-increasing, summable sequence $a=\{a_n\}$, define
\[
\displaystyle x_n(a)=\sum_{i=n}^{\infty}a_i
\quad \text{ and }\quad s_n(a)=2^{-n} x_{2^n}(a).\]

When the sequence $a$ is fixed and clear from context, we will simply write $s_n$ instead of $s_n (a)$.

Note that since $a$ is non increasing we have the estimate  
\[s_{n+1} < |I_{n,j}| < s_{n-1}
\quad \text{for all $1 \leq j \leq 2^n$,}\]
where $I_{n,j}$ are the closed intervals of step $n$ in the construction of the Cantor set $C_{a}$.

The Hausdorff dimension of sets in $\mathcal{C}_a$ was studied in \cite{BeyTa}, \cite{cabrelli2010classifying}, and \cite{garcia2007dimension}. The following proposition holds:

\begin{prop}\label{hausdorff cantor}
Let $a = \{a_n\}$ be a positive summable sequence, and let $C_a$ be the decreasing Cantor set associated with $a$. Then
\[
\dim_H C_{a} = \liminf_{n \to \infty} \frac{\log n}{-\log (x_n/n)}.
\]
Furthermore, any set $E \in \mathcal{C}_{a}$ satisfies $\dim_H E \leq \dim_H C_{a}$. Conversely, given any $s \in [0, \dim_H C_{a}]$, there exists a set $E \in \mathcal{C}_a$ such that $\dim_H E = s$.
\end{prop}

Moreover, it was shown in \cite[Proposition 3.3]{garcia2007dimension} that $\dim_H C_a = \lbd C_a$.

The box dimensions of sets in $\mathcal{C}_a$ were carefully studied in \cite[Chapter 3.2]{falconer1997techniques}. There, Falconer shows that the upper and lower box dimensions of sets in $\mathcal{C}_a$ depend only on the sequence $a$, and not on the rearrangement. That is, given any complementary set, its lower box dimension equals that of every other complementary set, and the same holds for the upper box dimension. Furthermore, for any $E \in \mathcal{C}_a$,
\[
\liminf_{n \to \infty} \frac{\log n}{-\log a_n} \leq \lbd E \leq \ubd E \leq \limsup_{n \to \infty} \frac{\log n}{-\log a_n},
\]
and the box dimension exists if and only if these limits coincide.

The upper box dimension of $C_a$ (and therefore of any other complementary set in $\mathcal{C}_a$) is given by 
\[
\limsup_{n \to \infty} \frac{\log n}{-\log (x_n/n)},
\]
which, as shown in \cite{tricot1994curves}, is equal to 
\[
\limsup_{n \to \infty} \frac{\log n}{-\log a_n}.
\]
This equivalence does not hold for the lower box dimension. In fact, the inequality
\[
\liminf_{n \to \infty} \frac{\log n}{-\log a_n} \leq \liminf_{n \to \infty} \frac{\log n}{-\log (x_n/n)}
\]
can be strict; see, for instance, \cite[Proposition 5]{cabrelli2004hausdorff}.
The Assouad dimension and the lower Assouad dimension of sets in \( \mathcal{C}_a \) were studied in \cite{garcia2018assouad}. There, the authors prove that, given a non-increasing summable sequence \( a = \{a_n\} \),
\[
\dim_A C_a = \limsup_{n \to \infty} \left( \sup_k \frac{n \log 2}{\log(s_k / s_{k+n})} \right),
\]
and
\[
\dim_L C_a = \liminf_{n \to \infty} \left( \inf_k \frac{n \log 2}{\log(s_k / s_{k+n})} \right).
\]
Moreover, it was shown that \( \dim_A D_a = 0 \) or \( 1 \), and if \( E \) is any complementary set in \( \mathcal{C}_a \), then
\[
\dim_A C_a \leq \dim_A E \leq \dim_A D_a,
\]
and
\[
0 = \dim_L D_a \leq \dim_L E \leq \dim_L C_a.
\]

Furthermore, if the sequence \( a \) is doubling, i.e., there exists a constant \( c > 0 \) such that \( a_n \leq c \, a_{2n} \) for all \( n \), then
\[
\{\dim_A E: E \in \mathcal{C}_a\} = [\dim_A C_a, 1],
\]
and if instead of doubling, we assume that \( \inf \frac{s_n}{s_{n+1}} > 2 \), then
\[
\{\dim_L E: E \in \mathcal{C}_a\} = [0, \dim_L C_a].
\]

For $\theta$-intermediate dimensions, the problem was studied recently by
Angelini and Molter~\cite{angelini2026intermediate}. They computed the
intermediate dimensions of $D_a$ and, under mild assumptions on $a$, those of
$C_a$ \cite[Theorems 4.2 and 4.3]{angelini2026intermediate}. They also proved
the following results for complementary sets. We state them only for the
upper intermediate dimensions; the analogous statements for the lower
intermediate dimensions also hold.

\begin{lem}\cite[Lemma 5.2 and Lemma 5.3]{angelini2026intermediate}\label{Maxmin}
    Let $a=\{a_n\}$ be a summable non-increasing sequence and let $E\in \mathcal{C}_a$ be any complementary set. Then
    $$\uda D_a \leq \uda E.$$
    Moreover, if the sequence satisfies $\sup_n s_n(a)/s_{n+1}(a)<\infty$ then
    $$\uda E \leq \uda C_a.$$
\end{lem}

\begin{theo}\cite[Theorem 6.1]{angelini2026intermediate}\label{attainablevaluesangelini}
    Let $\theta\in(0,1]$ and let $a=\{a_n\}$ be a positive, non-increasing, summable sequence. If
    $$2<\inf_n \frac{s_n(a)}{s_{n+1}(a)}\leq \sup_n \frac{s_n(a)}{s_{n+1}(a)}<\infty,$$
    then for every
    $$t\in[\uda D_a , \uda C_a],$$
    there exists a set $E=E(t,\theta)\in\mathcal{C}_a$ such that $\uda E=t$.
\end{theo}

The lower bound in Lemma~\ref{Maxmin} is given explicitly by the following
result.

\begin{theo}\cite[Theorem 4.3]{angelini2026intermediate}\label{dimDa}
    Let $a=\{a_n\}$ be a non-increasing summable sequence. Then, for all
    $\theta\in[0,1]$,
    $$\uda D_a=\frac{\theta\,\overline{\dim}_{\mathrm{B}} D_a}
    {1-(1-\theta)\,\overline{\dim}_{\mathrm{B}} D_a}.$$
\end{theo}

Since $\dim_{\mathrm{L}} D_a=0$ and $\dim_{\mathrm{A}} D_a=1$, this is the
smallest value allowed by Proposition~\ref{banajis bound}. Moreover, since the upper (and respectively lower) box dimension
is the same for every set in $\mathcal{C}_a$, the lower bound depends
only on $\overline{\dim}_{\mathrm{B}} C_a$.

The proof of Theorem~\ref{attainablevaluesangelini} is constructive, and it
gives more than the statement. Indeed, the set constructed there depends only
on a parameter $s>1$, and not on $\theta$.

\begin{rem}\label{hsfunctions}
    Let $a$ be as in Theorem~\ref{attainablevaluesangelini}, write
    $B:=\overline{\dim}_{\mathrm{B}} C_a$, and let $s>1$. The proof of
    \cite[Theorem 6.1]{angelini2026intermediate} provides a subsequence
    $b$ of $a$ with $s_n(b)\asymp s_n(a)^s$. In particular, $b$ satisfies
    $$2<\inf_n \frac{s_n(b)}{s_{n+1}(b)}\leq
    \sup_n \frac{s_n(b)}{s_{n+1}(b)}<\infty.$$
    It also provides a set $E_s=C_b\cup D'\in\mathcal{C}_a$, where $D'$ is
    a translate of $D_{a'}$ and $a'$ is the non-increasing rearrangement of
    the remaining terms of $a$. This set satisfies
    $$\uda E_s=h_s(\theta):=\max\Big\{\frac{\uda C_a}{s},\,
    \frac{\theta B}{1-(1-\theta)B}\Big\}
    \quad\text{for all }\theta\in[0,1].$$
\end{rem}

Thus the functions $h_s$, $s>1$, are realized as upper intermediate dimension
functions of complementary sets. By Lemma~\ref{Maxmin} and
Theorem~\ref{dimDa}, every such function $h$ satisfies
$$\frac{\theta B}{1-(1-\theta)B}\leq h(\theta)\leq \uda C_a .$$
This led to the question, raised in
\cite[Question 1]{angelini2026intermediate}, of describing the families
$$\mathcal{F}^U(a):=\{\theta\mapsto \uda E: E\in\mathcal{C}_a\}
\quad\text{and}\quad
\mathcal{F}^L(a):=\{\theta\mapsto \underline{\dim}_{\theta} E:
E\in\mathcal{C}_a\}.$$
The main result of this paper answers this question for $\mathcal{F}^U(a)$
under the hypothesis of Theorem~\ref{attainablevaluesangelini}.

For a recent survey of complementary sets and their dimensions, including developments beyond the scope of the present section, we refer the reader to \cite{survey}.

\section{Attainable dimension functions}\label{sec attainable functions}

In this section we prove the characterization of $\mathcal{F}^U(a)$
announced above. In Subsection~\ref{4.1} we recall the class of functions
$\mathcal{H}(\lambda,\alpha)$ and the results from \cite{Banaji-R22} that we
need. In Subsection~\ref{4.2} we reformulate the defining condition of
$\mathcal{H}(0,1)$ as the monotonicity of a quotient. This yields a sharp lower bound
for the functions in $\mathcal{F}^U(a)$ and a rigidity property. In
Subsections~\ref{4.3} and~\ref{4.4} we develop the two tools of the construction: the
homogeneous Cantor schemes, whose intermediate dimensions depend only on
their scales, and the mountains. The main result, Theorem~\ref{main functions} is proved in Subsection~\ref{4.5}.

Throughout this section we write
\[
B:=\ubd C_a \qquad\text{and}\qquad \varphi_B(\theta):=\frac{\theta B}{1-(1-\theta)B}\quad \text{for } \theta\in[0,1].
\]
Recall that $\ubd E=B$ for every $E\in\mathcal{C}_a$, moreover $\uda D_a=\varphi_B(\theta)$ by Theorem \cite[Theorem 4.3]{angelini2026intermediate}. As usual, we adopt the convention $\delta^{1/0}=0$ in Definition \ref{adef}, so that $\overline{\dim}_{\,0}E=\dimh E$ for every bounded set $E$. We will say that the sequence $a$ satisfies \emph{hypothesis} \eqref{hyp H} if
\begin{equation}\label{hyp H}\tag{H}
2<\kappa_1:=\inf_n \frac{s_n(a)}{s_{n+1}(a)}\leq \kappa_2:=\sup_n \frac{s_n(a)}{s_{n+1}(a)}<\infty .
\end{equation}

\subsection{Classes of functions and Banaji--Rutar results}\label{4.1}

Given a function $f$ defined on an interval, we write $D^+f$, $D_+f$, $D^-f$ and $D_-f$ for its upper right, lower right, upper left and lower left Dini derivatives, respectively. For instance,
\[
D^+f(x)=\limsup_{\epsilon\to 0^+}\frac{f(x+\epsilon)-f(x)}{\epsilon}
\]
and
\[
D_-f(x)=\liminf_{\epsilon\to 0^+}\frac{f(x)-f(x-\epsilon)}{\epsilon},
\]
whenever the corresponding one sided increments are defined.

\begin{defi}[{\cite[Definition 1.3]{Banaji-R22}}]\label{def H}
Let $0\leq\lambda\leq\alpha\leq 1$. If $\lambda<\alpha$, we denote by $\mathcal{H}(\lambda,\alpha)$ the set of all functions $h:[0,1]\to[\lambda,\alpha]$ which are non-decreasing, continuous on $(0,1]$, and satisfy
\begin{equation}\label{Dini bound}
D^+h(\theta)\leq \frac{(h(\theta)-\lambda)(\alpha-h(\theta))}{(\alpha-\lambda)\theta}\quad\text{for all } \theta\in(0,1).
\end{equation}
If $\lambda=\alpha$, we let $\mathcal{H}(\lambda,\alpha)$ consist only of the constant function $h\equiv\lambda$.
\end{defi}

The next theorem due to Banaji and Rutar provide a necessary and sufficient condition for a given function to be realized as the intermediate dimensions of a bounded subset $F\subset\mathbb{R}^d$. We write the case $d=1$.   

\begin{theo}[{\cite[Theorem B]{Banaji-R22}}]\label{BanajiRutar}
Let $F\subset \mathbb{R}$ be a bounded set with $\dim_L F=\lambda$ and $\dim_A F=\alpha$. Then with $h_1(\theta)=\da F$ and $h_2(\theta)=\uda F$ we have $h_1,h_2\in \mathcal{H}(\lambda,\alpha)$, $h_1\leq h_2$ and $h_1(0)=h_2(0)$.

Conversely, if $0\leq \lambda\leq \alpha\leq 1$ and $h_1,h_2\in \mathcal{H}(\lambda,\alpha)$ satisfy $h_1\leq h_2$ and $h_1(0)=h_2(0)$, then there exists a compact perfect set $F\subset \mathbb{R}$ such that $\dim_L F=\lambda$, $\dim_A F=\alpha$, $\da F=h_1(\theta)$ and $\uda F=h_2(\theta)$ for all $\theta\in[0,1]$.
\end{theo}

\begin{rmk}\label{inclusion H}
If $0\leq\lambda\leq\alpha\leq1$, then $\mathcal{H}(\lambda,\alpha)\subseteq\mathcal{H}(0,1)$. For constant functions this is clear. If $\lambda<\alpha$ and $h\in\mathcal{H}(\lambda,\alpha)$, write $u=h(\theta)-\lambda$ and $v=\alpha-h(\theta)$. The right hand side of \eqref{Dini bound} equals $\theta^{-1}(u^{-1}+v^{-1})^{-1}$ (with the value $0$ if $uv=0$), which is non-decreasing in $u$ and in $v$. Since $u\leq h(\theta)$ and $v\leq 1-h(\theta)$, we obtain
\[
D^+h(\theta)\leq \frac{h(\theta)(1-h(\theta))}{\theta}\quad\text{for all }\theta\in(0,1).
\]
Moreover, if $0<h(\theta)<1$ and $(\lambda,\alpha)\neq(0,1)$, then the right hand side of \eqref{Dini bound} is strictly smaller than $h(\theta)(1-h(\theta))/\theta$. In particular, the hypothesis $h\in\mathcal{H}(\lambda,\alpha)$ for some $0\leq\lambda<\alpha\leq1$ is equivalent to $h\in\mathcal{H}(0,1)$, i.e.
$$\bigcup_{0\leq \lambda\leq \alpha\leq 1}\mathcal{H}(\lambda,\alpha)= \mathcal{H}(0,1).$$
\end{rmk}

We also need some tools from \cite[Section 3]{Banaji-R22}, which we state only for $d=1$, $\lambda=0$ and $\alpha=1$. We denote by $\mathcal{G}(0,1)$ the set of continuous functions $g:\mathbb{R}\to[0,1]$ such that
\[
D^+g(x)\in[-g(x),\,1-g(x)]\quad\text{for all } x\in\mathbb{R},
\]
see \cite[Definition 3.1]{Banaji-R22}. Given $g\in\mathcal{G}(0,1)$ and $w\in\mathbb{R}$, the \emph{offset} $\kappa_w(g)\in\mathcal{G}(0,1)$ is defined by $\kappa_w(g)(x)=g(x-w)$ if $x\geq w$ and $\kappa_w(g)(x)=g(0)$ if $x\leq w$. A function $g\in\mathcal{G}(0,1)$ is \emph{rapidly decreasing} if there exist $y\in\mathbb{R}$ and $C>0$ such that $g(x)\leq g(y)\exp(y-x)+C\exp(-x)$ for all $x\geq y$. Note that a rapidly decreasing function satisfies $\lim_{x\to\infty}g(x)=0$.

Let $\rho=(\rho_k)_{k\geq0}$ be a strictly decreasing sequence of positive numbers converging to $0$. For $\phi\in(0,\rho_0)$ let $k=k_\rho(\phi)\geq1$ be the unique integer such that $\rho_k\leq \phi<\rho_{k-1}$, and set
\begin{equation}\label{def s rho}
s_\rho(\phi):=\frac{k_\rho(\phi)\log 2}{-\log\phi}.
\end{equation}
Observe that $k_\rho(\phi)=\min\{k\geq0 : \rho_k\leq\phi\}$. If $\mathbf{r}=(r_j)_{j\geq1}\subset(0,1/2]$, we write $s_{\mathbf r}:=s_\rho$ where $\rho_0=1$ and $\rho_k=r_1\cdots r_k$. This is exactly the function $s_{\mathbf r}(\delta)$ of \cite[Section 2.3]{Banaji-R22}.

\begin{lem}[{\cite[Lemma 3.4]{Banaji-R22}}]\label{BR lemma 3.4}
Let $\tilde g\in\mathcal{G}(0,1)$ be a function which is not rapidly decreasing and such that $0<\tilde g(0)<1$. Then there is a constant $w_0\in\mathbb{R}$, depending only on $\tilde g(0)$, such that for every $w\geq w_0$ there exists a sequence $\mathbf r=(r_j)_{j\geq1}\subset(0,1/2]$ for which $g:=\kappa_w(\tilde g)$ satisfies
\[
\bigl|s_{\mathbf r}(\exp(-\exp(x)))-g(x)\bigr|\leq \log 2\,\exp(-x)\quad\text{for all } x\geq w_0 .
\]
\end{lem}

Finally, we recall two elementary facts on Dini derivatives.

\begin{lem}[{\cite[Lemma 2.2]{Banaji-R22}}]\label{BR lemma 2.2}
Let $f$ and $g$ be continuous functions on $[a,b]$ with $D_+g\leq D^+f$ on $[a,b)$ and $g(a)=f(a)$. Then $g\leq f$ on $[a,b]$.
\end{lem}

\begin{lem}[{\cite[Lemma 2.4]{Banaji-R22}}]\label{BR lemma 2.4}
Let $g:\mathbb{R}\to(0,1)$ be continuous and let $V\subset\mathbb{R}$ be open. Then $D^+g(x)\in[-g(x),1-g(x)]$ for all $x\in V$ if and only if $D_-g(x)\in[-g(x),1-g(x)]$ for all $x\in V$.
\end{lem}

\subsection{An integrated form of the Dini condition}\label{4.2}

The following elementary observation is the key tool of this section. 
Condition \eqref{Dini bound} is a pointwise inequality for a Dini derivative, which is
awkward to manipulate directly. The following lemma replaces it, for
$\lambda=0$ and $\alpha=1$, by an equivalent global statement: the quotient
$Q_h(\theta)=h(\theta)/(\theta(1-h(\theta)))$ is non-increasing. This simple
reformulation is used repeatedly in the rest of the section. First, comparing
$Q_h$ with the constant quotient of $\varphi_B$ yields the lower bound
$h\ge\varphi_B$ and its rigidity (Proposition~\ref{lower bound rigidity}). In the proof of the main
Theorem~\ref{main functions}, this bound is what allows the countable part of the
construction to be added without increasing the dimension. Second, the
monotonicity of $Q_h$ is exactly what makes the mountains of Lemma~\ref{mountain} well
defined, and it ensures that their slopes are admissible. These mountains are
the building blocks of the set realizing a prescribed function in
Theorem~\ref{main functions}.

\begin{lem}\label{monotone quotient}
Let $h\in\mathcal{H}(0,1)$ with $0<h(1)<1$. Then $0<h(\theta)<1$ for every $\theta\in(0,1]$, and the function
\[
Q_h(\theta):=\frac{h(\theta)}{\theta\,(1-h(\theta))}
\]
is non-increasing on $(0,1]$.
\end{lem}

\begin{proof}
Since $h$ is non-decreasing, $h(\theta)\leq h(1)<1$ for all $\theta$. Let $0<c<1$ be such that $h>0$ on $[c,1]$, and define on $[c,1]$ the continuous function
\[
P(\theta):=\log\frac{h(\theta)}{1-h(\theta)}-\log\theta .
\]
The map $v\mapsto \log (v/(1-v))$ is continuously differentiable on $(0,1)$ with derivative $1/(v(1-v))>0$. Since $h$ is continuous at each $\theta\in[c,1)$ and $0\leq D^+h(\theta)\leq h(\theta)(1-h(\theta))/\theta<\infty$, the mean value theorem gives
\[
D^+P(\theta)=\frac{D^+h(\theta)}{h(\theta)(1-h(\theta))}-\frac{1}{\theta}\leq 0\qquad\text{for all }\theta\in[c,1).
\]
Applying Lemma \ref{BR lemma 2.2} to $g=P$ and to the constant function $f\equiv P(c')$ on $[c',1]$, for any $c'\in[c,1)$, we obtain $P\leq P(c')$ on $[c',1]$. Hence $P$ is non-increasing on $[c,1]$.

Suppose now that $h(\theta_0)=0$ for some $\theta_0\in(0,1]$. Then $\theta_0<1$, and since $h$ is continuous on $(0,1]$, the set $\{\theta\in[\theta_0,1] : h(\theta)=0\}$ is compact and does not contain $1$. Let $\theta_1<1$ be its maximum. Then $h>0$ on $(\theta_1,1]$, so by the previous paragraph $P(\theta)\geq P(1)>-\infty$ for every $\theta\in(\theta_1,1]$. On the other hand $h(\theta)\to h(\theta_1)=0$ as $\theta\to\theta_1^+$, which gives $P(\theta)\to-\infty$, a contradiction. Therefore $h>0$ on $(0,1]$, the function $P$ is non-increasing on $(0,1]$, and so is $Q_h=\exp(P)$.
\end{proof}

\begin{prop}\label{lower bound rigidity}
Let $h\in\mathcal{H}(0,1)$ with $h(1)=B\in(0,1)$. Then the following hold.
\begin{enumerate}
\item[(i)] $h(\theta)\geq\varphi_B(\theta)$ for all $\theta\in[0,1]$.
\item[(ii)] If $h(\theta_0)=\varphi_B(\theta_0)$ for some $\theta_0\in(0,1]$, then $h(\theta)=\varphi_B(\theta)$ for all $\theta\in[\theta_0,1]$.
\end{enumerate}
\end{prop}

\begin{proof}
A direct computation shows that $\varphi_B(\theta)/(1-\varphi_B(\theta))=\theta B/(1-B)$, that is, $Q_{\varphi_B}(\theta)=B/(1-B)$ for all $\theta\in(0,1]$. By Lemma \ref{monotone quotient},
\[
Q_h(\theta)\geq Q_h(1)=\frac{B}{1-B}=Q_{\varphi_B}(\theta)\qquad\text{for all }\theta\in(0,1].
\]
Since $v\mapsto v/(1-v)$ is strictly increasing on $[0,1)$, the inequality $Q_h(\theta)\geq Q_{\varphi_B}(\theta)$ is equivalent to $h(\theta)\geq\varphi_B(\theta)$. For $\theta=0$ we simply have $h(0)\geq0=\varphi_B(0)$. This proves (i). If $h(\theta_0)=\varphi_B(\theta_0)$, then $Q_h(\theta_0)=Q_h(1)$, and since $Q_h$ is non-increasing, $Q_h$ is constant on $[\theta_0,1]$. Thus $Q_h=Q_{\varphi_B}$ on $[\theta_0,1]$, which means $h=\varphi_B$ on $[\theta_0,1]$, proving (ii).
\end{proof}

\begin{rmk}\label{rmk rigidity}
Part (i) of Proposition \ref{lower bound rigidity} recovers the lower bound $\uda D_a\leq\uda E$ of Lemma \ref{Maxmin}. Part (ii) shows that this lower bound is rigid. If the upper intermediate dimension of a complementary set coincides with $\varphi_B$ at a single $\theta_0\in(0,1]$, then it coincides with $\varphi_B=\overline{\dim}_{\,\theta} D_a$ on the whole interval $[\theta_0,1]$. The equality cannot be propagated to the left of $\theta_0$ in general. For instance, if $a$ satisfies \eqref{hyp H} and $s>1$, the proof of Theorem \ref{attainablevaluesangelini} provides $E\in\mathcal{C}_a$ with
\[
\uda E=h_s(\theta):=\max\Bigl\{\frac{\uda C_a}{s},\ \varphi_B(\theta)\Bigr\}\quad\text{for all }\theta\in[0,1].
\]
Since $\overline{\dim}_{\,1}C_a/s=B/s<B=\varphi_B(1)$ and both functions are continuous at $\theta=1$, the function $h_s$ coincides with $\varphi_B$ on a neighbourhood of $\theta=1$, while at $\theta=0$ it equals $\dimh C_a/s>0=\varphi_B(0)$.
\end{rmk}

The equality case in Proposition \ref{lower bound rigidity} also forces the Assouad and lower dimensions to be extremal.

\begin{prop}\label{extremal assouad}
Let $F\subset\mathbb{R}$ be a bounded set with $\ubd F=B\in(0,1)$. If $\overline{\dim}_{\,\theta_0}F=\varphi_B(\theta_0)$ for some $\theta_0\in(0,1)$, then $\dim_L F=0$ and $\dim_A F=1$.
\end{prop}

\begin{proof}
Write $\lambda=\dim_L F$ and $\alpha=\dim_A F$, so that $\lambda\leq B\leq\alpha$. If $\lambda=\alpha$, then $\theta\mapsto\uda F$ is constant by Theorem \ref{BanajiRutar}, hence equal to $B$, and $B>\varphi_B(\theta_0)$ because $0<B<1$ and $\theta_0<1$. Thus $\lambda<\alpha$, and Proposition \ref{banajis bound} gives $\varphi_B(\theta_0)=\overline{\dim}_{\,\theta_0}F\geq\psi(\theta_0)$, where
\[
\psi(\theta):=\frac{\theta\alpha(B-\lambda)+\lambda(\alpha-B)}{\theta(B-\lambda)+(\alpha-B)}.
\]
Set $u:=B-\lambda\in[0,B]$ and $v:=\alpha-B\in[0,1-B]$. A direct computation gives
\[
\psi(\theta)=B-\frac{(1-\theta)\,uv}{v+\theta u}
\qquad\text{and}\qquad
\varphi_B(\theta)=B-\frac{(1-\theta)\,B(1-B)}{(1-B)+\theta B}.
\]
If $uv=0$, then $\psi(\theta_0)=B>\varphi_B(\theta_0)$, which is impossible. If $uv>0$, then $uv/(v+\theta_0 u)=(u^{-1}+\theta_0v^{-1})^{-1}$ is strictly increasing in $u$ and in $v$, so $\psi(\theta_0)\geq\varphi_B(\theta_0)$ with equality if and only if $u=B$ and $v=1-B$. Hence $\lambda=0$ and $\alpha=1$.
\end{proof}

\subsection{Homogeneous Cantor schemes}\label{4.3}

The set realizing a prescribed function in Theorem~\ref{main functions} cannot, in general,
be a decreasing Cantor set (see Remark~\ref{rmk why schemes}). We therefore work with a more
flexible class of sets. These are limit sets of binary Cantor schemes in
which the intervals of a given level have comparable lengths, which need not
be equal. Their gaps also need not be arranged in decreasing order. The main
result of this subsection, Lemma~\ref{dimension scheme}, shows that for such sets the upper
intermediate dimensions depend only on the sequence of scales $\rho$. They
are given by window infima of the function $S_\rho$ defined below. This reduces the
construction of Theorem~\ref{main functions} to two tasks: choosing a suitable sequence of
scales, and placing terms of $a$ as gaps so that the resulting scheme is
homogeneous with those scales.

Let $\Sigma^*=\bigcup_{k\geq0}\{1,2\}^k$ be the set of finite words over the alphabet $\{1,2\}$, including the empty word $\varnothing$, and write $|\sigma|$ for the length of $\sigma$. We write $\tau\succeq\sigma$ if $\sigma$ is a prefix of $\tau$ (including $\tau=\sigma$).

\begin{defi}\label{def scheme}
A family $\{I_\sigma\}_{\sigma\in\Sigma^*}$ of non-degenerate closed intervals is a \emph{binary Cantor scheme} if for every $\sigma\in\Sigma^*$ the intervals $I_{\sigma1}$ and $I_{\sigma2}$ are disjoint and contained in $I_\sigma$, and $\max_{|\sigma|=k}|I_\sigma|\to0$ as $k\to\infty$. Its \emph{limit set} is
\[
F=\bigcap_{k\geq0}\ \bigcup_{|\sigma|=k}I_\sigma .
\]
Let $\rho=(\rho_k)_{k\geq0}$ be a strictly decreasing sequence converging to $0$ with $\rho_0\leq1$, and let $K\geq1$. The scheme is \emph{$(\rho,K)$-homogeneous} if $K^{-1}\rho_k\leq|I_\sigma|\leq K\rho_k$ for all $\sigma\in\Sigma^*$ with $|\sigma|=k$ and all $k\geq0$.
\end{defi}

For instance, the closed intervals $I_{k,j}$ used in Section \ref{Comp sets} to construct $C_a$ form, after the obvious relabelling by words, a binary Cantor scheme with limit set $C_a$, and $s_{k+1}(a)\leq|I_{k,j}|\leq s_{k-1}(a)$. Therefore, if $a$ satisfies \eqref{hyp H} and $C\geq1$, this scheme is $(\rho,K)$-homogeneous with $\rho_k=C^{-1}s_k(a)$ and $K=C\kappa_2$.

For a strictly decreasing sequence $\rho$ as in Definition \ref{def scheme} we will also use the function
\[
S_\rho(x):=s_\rho\bigl(\exp(-\exp(x))\bigr),\qquad x>\log\log(1/\rho_0).
\]
Since $\phi=\exp(-\exp(y))$ is a decreasing bijection, $\phi\in[\delta^{1/\theta},\delta]$ if and only if $y\in[x,x+\log(1/\theta)]$, where $x=\log\log(1/\delta)$. The next lemma is essentially \cite[Lemma 2.10 and Proposition 2.11]{Banaji-R22}, where it is proved for homogeneous Moran sets. We include the proof since in our setting the intervals of a fixed level need not have the same length.

\begin{lem}\label{dimension scheme}
Let $F$ be the limit set of a $(\rho,K)$-homogeneous binary Cantor scheme. Then for all $\theta\in(0,1]$,
\begin{equation}\label{formula scheme}
\uda F=\limsup_{\delta\to0}\ \inf_{\phi\in[\delta^{1/\theta},\delta]}s_\rho(\phi)
=\limsup_{x\to\infty}\ \inf_{y\in[x,x+\log(1/\theta)]}S_\rho(y),
\end{equation}
and
\begin{equation}\label{formula scheme H}
\dimh F=\lbd F=\liminf_{\phi\to0}s_\rho(\phi)=\liminf_{x\to\infty}S_\rho(x).
\end{equation}
\end{lem}

\begin{proof}
The second equalities in \eqref{formula scheme} and \eqref{formula scheme H} follow from the change of variables described above, so we only prove the first ones. Let $\mu$ be the image of the uniform Bernoulli measure on $\{1,2\}^{\mathbb{N}}$ under the coding map which sends $\omega$ to the unique point of $\bigcap_k I_{\omega|k}$. Since the intervals of each level are pairwise disjoint, $\mu(I_\sigma)=2^{-|\sigma|}$ for every $\sigma$, and $\mu(F)=1$.

We first prove two estimates. Let $U\subset\mathbb{R}$ with $0<|U|<\rho_0$ and let $k=k_\rho(|U|)\geq1$, so that $\rho_k\leq|U|<\rho_{k-1}$. The intervals $I_\sigma$ with $|\sigma|=k-1$ are pairwise disjoint and have length at least $K^{-1}\rho_{k-1}>K^{-1}|U|$. If $U$ intersects $N$ of them, at least $N-2$ of them are contained in the convex hull of $U$, so $(N-2)K^{-1}\rho_{k-1}\leq|U|<\rho_{k-1}$, that is, $N<K+2$. Since $F\subset\bigcup_{|\sigma|=k-1}I_\sigma$, we obtain
\begin{equation}\label{mass estimate}
\mu(U)\leq (K+2)\,2^{-(k-1)}=2(K+2)\,|U|^{s_\rho(|U|)},
\end{equation}
where we used that $|U|^{s_\rho(|U|)}=\exp(-k\log2)=2^{-k}$ by \eqref{def s rho}. On the other hand, if $\phi\in(0,\rho_0)$ and $k=k_\rho(\phi)$, the $2^k$ intervals $I_\sigma$ with $|\sigma|=k$ cover $F$ and have length at most $K\rho_k\leq K\phi$. Covering each of them by $\lceil K\rceil$ intervals of length $\phi$ we obtain a cover $\mathcal{V}_\phi$ of $F$ by $\lceil K\rceil 2^k$ sets of diameter $\phi$, and for every $t\geq0$,
\begin{equation}\label{cover estimate}
\sum_{V\in\mathcal{V}_\phi}|V|^t=\lceil K\rceil\,2^k\phi^t=\lceil K\rceil\,\phi^{\,t-s_\rho(\phi)} .
\end{equation}

Fix $\theta\in(0,1]$ and let $d$ denote the $\limsup$ in \eqref{formula scheme}. Let $t>t'>d$. There is $\delta_0\in(0,\rho_0)$ such that for all $\delta\in(0,\delta_0)$ there exists $\phi\in[\delta^{1/\theta},\delta]$ with $s_\rho(\phi)<t'$. By \eqref{cover estimate}, $\mathcal{V}_\phi$ is a cover of $F$ by sets with diameters in $[\delta^{1/\theta},\delta]$ and
\[
\sum_{V\in\mathcal{V}_\phi}|V|^t=\lceil K\rceil\,\phi^{\,t-s_\rho(\phi)}\leq\lceil K\rceil\,\delta^{\,t-t'}\xrightarrow[\delta\to0]{}0 .
\]
Hence $\uda F\leq t$, and letting $t\to d$ gives $\uda F\leq d$.

Conversely, suppose $d>0$ and let $0<t<d$. There exist $\delta_n\to0$ with $\delta_n<\rho_0$ such that $s_\rho(\phi)>t$ for all $\phi\in[\delta_n^{1/\theta},\delta_n]$. Let $\{U_i\}$ be any cover of $F$ with $\delta_n^{1/\theta}\leq|U_i|\leq\delta_n$, where we may assume that each $U_i$ intersects $F$. By \eqref{mass estimate}, and since $|U_i|<1$,
\[
1=\mu(F)\leq\sum_i\mu(U_i)\leq2(K+2)\sum_i|U_i|^{s_\rho(|U_i|)}\leq2(K+2)\sum_i|U_i|^{t}.
\]
Thus every such cover has $t$-cost at least $1/(2(K+2))$, for arbitrarily small $\delta_n$, and by Definition \ref{adef} we get $\uda F\geq t$. Letting $t\to d$ proves \eqref{formula scheme}.

Finally, if $t<\liminf_{\phi\to0}s_\rho(\phi)$, then \eqref{mass estimate} gives $\mu(U)\leq2(K+2)|U|^t$ for every set $U$ of sufficiently small diameter, and the mass distribution principle (see \cite{FBook}) yields $\dimh F\geq t$. On the other hand, \eqref{cover estimate} shows that $F$ can be covered by $\lceil K\rceil2^{k_\rho(\phi)}$ sets of diameter $\phi$, so
\[
\lbd F\leq\liminf_{\phi\to0}\frac{\log\bigl(\lceil K\rceil2^{k_\rho(\phi)}\bigr)}{-\log\phi}=\liminf_{\phi\to0}s_\rho(\phi).
\]
Since $\dimh F\leq\lbd F$, \eqref{formula scheme H} follows.
\end{proof}

\begin{rmk}\label{rmk formula Ca}
If $a$ satisfies \eqref{hyp H}, Lemma \ref{dimension scheme} applied to the scheme defining $C_a$ gives an alternative expression for $\uda C_a$, namely
\[
\uda C_a=\limsup_{x\to\infty}\ \inf_{y\in[x,x+\log(1/\theta)]}S_{\rho}(y),\qquad \rho_k=C^{-1}s_k(a),
\]
for any constant $C\geq1$, which complements Theorem \cite[Theorem 4.2]{angelini2026intermediate}.
\end{rmk}

We also record some elementary consequences of \eqref{hyp H}.

\begin{lem}\label{consequences H}
Let $a$ be a positive, non-increasing sequence with $\sum a_n=1$ satisfying \eqref{hyp H}. Then the following hold.
\begin{enumerate}
\item[(a)] $(1-2/\kappa_1)\,s_k(a)\leq a_{2^k}\leq s_{k-1}(a)\leq\kappa_2\,s_k(a)$ for all $k\geq1$.
\item[(b)] $a_{2^k}\leq c_a\,a_{2^{k+1}}$ for all $k\geq1$, where $c_a:=\kappa_1\kappa_2^2/(\kappa_1-2)$.
\item[(c)] The Cantor set $C_a$ satisfies
\[
0<\frac{\log2}{\log\kappa_2}\leq\dim_L C_a\leq\dimh C_a\leq B\leq\dim_A C_a\leq\frac{\log2}{\log\kappa_1}<1 .
\]
\end{enumerate}
\end{lem}

\begin{proof}
Write $s_k=s_k(a)$. Since $a$ is non-increasing and $s_k=2^{-k}\sum_{i=2^k}^{2^{k+1}-1}a_i+2s_{k+1}$, we have
\[
a_{2^k}\geq2^{-k}\sum_{i=2^k}^{2^{k+1}-1}a_i=s_k-2s_{k+1}\geq\Bigl(1-\frac{2}{\kappa_1}\Bigr)s_k .
\]
Moreover $a_{2^k}\leq a_i$ for $2^{k-1}\leq i<2^k$, so $a_{2^k}\leq2^{-(k-1)}\sum_{i=2^{k-1}}^{2^k-1}a_i\leq s_{k-1}\leq\kappa_2s_k$. This proves (a). By (a),
\[
a_{2^k}\leq\kappa_2s_k\leq\kappa_2^2s_{k+1}\leq\frac{\kappa_2^2}{1-2/\kappa_1}\,a_{2^{k+1}}=c_a\,a_{2^{k+1}},
\]
which is (b). Finally, \eqref{hyp H} gives $\kappa_1^n\leq s_k/s_{k+n}\leq\kappa_2^n$ for all $k$ and $n$, so every quotient $n\log2/\log(s_k/s_{k+n})$ lies in $[\log2/\log\kappa_2,\log2/\log\kappa_1]$. Part (c) now follows from the formulas for $\dim_A C_a$ and $\dim_L C_a$ recalled in Section \ref{Comp sets} and from the inequalities $\dim_L\leq\dimh\leq\lbd\leq\ubd\leq\dim_A$.
\end{proof}

\subsection{Mountains}\label{4.4}

By Lemma~\ref{dimension scheme}, prescribing the upper intermediate dimensions of a set
amounts to constructing a function $G$, playing the role of $S_\rho$, whose
window infima over intervals of length $\log(1/\theta)$ reproduce $h(\theta)$. In addition, the function must belong to the class $\mathcal{G}(0,1)$; that
is, its Dini derivatives must lie between $-G$ and $1-G$. These are the
slopes that a function of the form $S_\rho$ can have. This condition allows
Lemma~\ref{BR lemma 3.4} to produce a sequence of scales $\rho$ with $S_\rho$ close to the
constructed function. The basic building block of this
function is a mountain associated with $h$ and a parameter $\varepsilon>0$.
This is a continuous function with admissible slopes which starts and ends
at $h(\varepsilon)$. For every $\theta\in[\varepsilon,1]$, it contains a
window of length $\log(1/\theta)$ on which it stays above $h(\theta)$, while
every such window meets a point where it is at most $h(\theta)$. In the
proof of Theorem~\ref{main functions}, infinitely many mountains, with
$\varepsilon\to 0$, are joined by valleys to produce the required function.

The next lemma is a version of the \emph{mountains} of \cite[Theorem 3.9]{Banaji-R22} in the case $\lambda=0$ and $\alpha=1$. We present it in a slightly different way, based on Lemma \ref{monotone quotient}, which makes all the required properties transparent.

\begin{lem}\label{mountain}
Let $h\in\mathcal{H}(0,1)$ with $h(1)=B\in(0,1)$, let $\epsilon\in(0,1)$ and write $\ell_\epsilon:=\log(1/\epsilon)$. There exists a continuous function $f=f_{h,\epsilon}:[0,\ell_\epsilon]\to[h(\epsilon),B]$ with the following properties.
\begin{enumerate}
\item[(M1)] $f(0)=f(\ell_\epsilon)=h(\epsilon)$ and $\max_{[0,\ell_\epsilon]}f=B$.
\item[(M2)] $D_-f(y)\in[-f(y),\,1-f(y)]$ for every $y\in(0,\ell_\epsilon]$.
\item[(M3)] For every $\theta\in[\epsilon,1]$ there exists $x_\theta\in[0,\ell_\epsilon-\log(1/\theta)]$ such that $f(y)\geq h(\theta)$ for all $y\in[x_\theta,x_\theta+\log(1/\theta)]$.
\item[(M4)] For every $\theta\in[\epsilon,1]$ and every $y\in[0,\ell_\epsilon-\log(1/\theta)]$ there exists $z\in[y,y+\log(1/\theta)]$ with $f(z)\leq h(\theta)$.
\end{enumerate}
\end{lem}

\begin{proof}
By Lemma \ref{monotone quotient}, $0<h<1$ on $(0,1]$ and $Q_h$ is non-increasing. For $\theta\in[\epsilon,1]$ define
\[
x_L(\theta):=\log\frac{1-h(\epsilon)}{1-h(\theta)}\qquad\text{and}\qquad x_R(\theta):=x_L(\theta)+\log\frac1\theta .
\]
The function $x_L$ is continuous and non-decreasing, with $x_L(\epsilon)=0$ and $x_L(1)=x^*:=\log\bigl((1-h(\epsilon))/(1-B)\bigr)$. Since $1/(1-h(\theta))=1+\theta Q_h(\theta)$, we can write
\[
x_R(\theta)=\log\bigl(1-h(\epsilon)\bigr)+\log\Bigl(\frac1\theta+Q_h(\theta)\Bigr),
\]
and since $Q_h$ is non-increasing while $1/\theta$ is strictly decreasing, $x_R$ is continuous and strictly decreasing on $[\epsilon,1]$, with $x_R(\epsilon)=\ell_\epsilon$ and $x_R(1)=x^*$. In particular $0\leq x^*\leq\ell_\epsilon$. Let $\tau:[x^*,\ell_\epsilon]\to[\epsilon,1]$ be the inverse of $x_R$, which is continuous and strictly decreasing, and define
\[
f(y):=
\begin{cases}
1-(1-h(\epsilon))\,e^{-y} & \text{if } y\in[0,x^*],\\[2pt]
h(\tau(y)) & \text{if } y\in[x^*,\ell_\epsilon].
\end{cases}
\]
Both expressions are equal to $B$ at $y=x^*$, so $f$ is continuous. On $[0,x^*]$ the function $f$ increases from $h(\epsilon)$ to $B$, and on $[x^*,\ell_\epsilon]$ it is non-increasing, from $B$ to $h(\tau(\ell_\epsilon))=h(\epsilon)$. This gives (M1) and shows that $f$ takes values in $[h(\epsilon),B]$. Moreover, for every $\theta\in[\epsilon,1]$,
\begin{equation}\label{mountain identity}
f(x_L(\theta))=h(\theta)=f(x_R(\theta)).
\end{equation}

To prove (M3), let $\theta\in[\epsilon,1]$ and put $x_\theta:=x_L(\theta)$, so that $x_\theta+\log(1/\theta)=x_R(\theta)\leq\ell_\epsilon$. If $y\in[x_L(\theta),x^*]$, then $f(y)\geq f(x_L(\theta))=h(\theta)$ because $f$ is increasing on $[0,x^*]$. If $y\in[x^*,x_R(\theta)]$, then $\tau(y)\geq\theta$ and so $f(y)=h(\tau(y))\geq h(\theta)$.

To prove (M4), let $\theta\in[\epsilon,1]$ and $y\in[0,\ell_\epsilon-\log(1/\theta)]$. If $y\leq x_L(\theta)$, then $f(y)\leq f(x_L(\theta))=h(\theta)$. Otherwise $z:=y+\log(1/\theta)$ satisfies $x^*\leq x_R(\theta)<z\leq\ell_\epsilon$, so $\tau(z)<\theta$ and $f(z)=h(\tau(z))\leq h(\theta)$. In both cases we have found a point of $[y,y+\log(1/\theta)]$ at which $f$ is at most $h(\theta)$.

It remains to prove (M2). For $y\in(0,x^*]$ we have $D_-f(y)=(1-h(\epsilon))e^{-y}=1-f(y)\in[-f(y),1-f(y)]$. On $(x^*,\ell_\epsilon]$ the function $f$ is non-increasing, so $D_-f(y)\leq0\leq1-f(y)$. For the lower bound we claim that
\begin{equation}\label{slow descent}
f(y')\leq f(y)\,e^{\,y-y'}\qquad\text{whenever } x^*\leq y'\leq y\leq\ell_\epsilon .
\end{equation}
Indeed, let $\theta=\tau(y)\leq\theta'=\tau(y')$. From the definition of $Q_h$ we get $h=Q_h/(\theta^{-1}+Q_h)$, hence
\[
f(y)=\frac{Q_h(\theta)}{\theta^{-1}+Q_h(\theta)},\qquad f(y')=\frac{Q_h(\theta')}{\theta'^{-1}+Q_h(\theta')},
\]
and
\[
e^{\,y-y'}=e^{\,x_R(\theta)-x_R(\theta')}=\frac{\theta^{-1}+Q_h(\theta)}{\theta'^{-1}+Q_h(\theta')}.
\]
Therefore $f(y)e^{\,y-y'}=Q_h(\theta)/(\theta'^{-1}+Q_h(\theta'))\geq Q_h(\theta')/(\theta'^{-1}+Q_h(\theta'))=f(y')$, since $Q_h(\theta)\geq Q_h(\theta')$. Now if $y\in(x^*,\ell_\epsilon]$ and $0<t<y-x^*$, then \eqref{slow descent} gives
\[
\frac{f(y)-f(y-t)}{t}\geq f(y)\,\frac{1-e^{t}}{t}\xrightarrow[t\to0^+]{}-f(y),
\]
so $D_-f(y)\geq-f(y)$. This completes the proof.
\end{proof}

\begin{rmk}\label{rmk mountain}
The left side of the mountain increases at the maximal rate allowed in $\mathcal{G}(0,1)$. By \eqref{slow descent}, the right side decreases at the maximal rate exactly on the intervals where $Q_h$ is constant, that is, where $h$ is a solution of the differential equation $h'=h(1-h)/\theta$. This is the case, for instance, on every interval where $h=\varphi_B$.
\end{rmk}

\subsection[Prescribing the upper intermediate dimensions]{Prescribing the upper intermediate dimensions of complementary sets}\label{4.5}

We can now state and prove the main result of this section.

\begin{theo}\label{main functions}
Let $a=\{a_n\}$ be a positive, non-increasing sequence with $\sum_n a_n=1$ such that
\[
2<\inf_n\frac{s_n(a)}{s_{n+1}(a)}\leq\sup_n\frac{s_n(a)}{s_{n+1}(a)}<\infty .
\]
Let $0\leq\lambda<\alpha\leq1$ and let $h\in\mathcal{H}(\lambda,\alpha)$ be such that
\[
h(\theta)\leq\uda C_a\ \ \text{for all }\theta\in[0,1]\qquad\text{and}\qquad h(1)=\ubd C_a .
\]
Then there exists $E\in\mathcal{C}_a$ such that $\uda E=h(\theta)$ for all $\theta\in[0,1]$.
\end{theo}

Before giving the proof, let us describe the strategy. As in the proof of Theorem \ref{attainablevaluesangelini}, the set $E$ will be the union of a Cantor type set $M$, built with some of the terms of $a$, and a countable set $D'$, built with the remaining terms. The set $M$ will satisfy $\uda M=h(\theta)$, while $\uda D'\leq\varphi_B(\theta)\leq h(\theta)$ by Proposition \ref{lower bound rigidity}. To construct $M$ we follow the approach of \cite{Banaji-R22}. We concatenate mountains associated with $h$ and valleys to obtain a function $G\in\mathcal{G}(0,1)$ whose window infima reproduce $h$, and Lemma \ref{BR lemma 3.4} provides a sequence of scales $\rho^{(2)}$ adapted to $G$. Two new difficulties appear. First, since $M$ must be built with terms of $a$, the lengths of its intervals of level $k$ should not exceed, up to a constant factor, those of $C_a$ (see Remark \ref{rmk minimum}), while the hypothesis $h\leq\uda C_a$ only gives information about infima over windows. We overcome this by taking the minimum of $\rho^{(2)}$ and the scales of $C_a$, and by placing each mountain on a window where the covering numbers of $C_a$ are large. Second, the function $h$ may saturate \eqref{Dini bound} for $\lambda=0$ and $\alpha=1$, so the ratios of consecutive scales may approach $0$ and $1/2$. For this reason $M$ is not a decreasing Cantor set $C_b$ as in Theorem \ref{attainablevaluesangelini}, but the limit set of a homogeneous Cantor scheme in which the gaps are not arranged in decreasing order (see Remark \ref{rmk why schemes}).

\begin{proof}
\emph{Step 0. Preliminary reductions.} By Remark \ref{inclusion H}, $h\in\mathcal{H}(0,1)$. By Lemma \ref{consequences H}(c), $B=\ubd C_a\in(0,1)$, and then Lemma \ref{monotone quotient} and Proposition \ref{lower bound rigidity} give
\begin{equation}\label{h positive}
0<h(\theta)\leq B<1\quad\text{for }\theta\in(0,1]\qquad\text{and}\qquad h\geq\varphi_B\ \text{ on }[0,1].
\end{equation}
We write $U(\theta):=\uda C_a$ for $\theta\in[0,1]$, so that $U(0)=\dimh C_a$.

\emph{Step 1. The scales of $C_a$.} Let $\kappa_1,\kappa_2$ be as in \eqref{hyp H}, let $c_a$ be as in Lemma \ref{consequences H}, and set
\[
C:=\frac{\kappa_1\kappa_2}{\kappa_1-2}>1\qquad\text{and}\qquad \rho^{(1)}_k:=C^{-1}s_k(a)\quad(k\geq0),
\]
and write $S_1:=S_{\rho^{(1)}}$.
As observed after Definition \ref{def scheme}, $C_a$ is the limit set of a $(\rho^{(1)},C\kappa_2)$-homogeneous binary Cantor scheme. Hence Lemma \ref{dimension scheme} gives
\begin{equation}\label{U window}
\begin{split}
U(\theta)&=\limsup_{x\to\infty}\ \inf_{y\in[x,x+\log(1/\theta)]}S_1(y)\qquad\text{for }\theta\in(0,1],\\
U(0)&=\liminf_{x\to\infty}S_1(x).
\end{split}
\end{equation}

\emph{Step 2. Parameters.} Let $\{q_1,q_2,\dots\}$ be an enumeration of $\mathbb{Q}\cap(0,1)$ and let $(\theta_j)_{j\geq1}$ be the sequence
\[
q_1,\ q_1,q_2,\ q_1,q_2,q_3,\ q_1,q_2,q_3,q_4,\ \dots,
\]
in which every $q_i$ appears infinitely many times. Put
\[
\epsilon_j:=\min\Bigl\{\theta_1,\dots,\theta_j,\frac{1}{j+1}\Bigr\}.
\]
Then the sequence $(\epsilon_j)$ is non-increasing and converges to $0$, and $0<\epsilon_j\leq\theta_j<1$ for every $j$. By \eqref{h positive}, $h(\epsilon_j)>0$, and we define
\[
\beta_j:=\max\Bigl\{h(0),\ \frac{h(\epsilon_{j+1})}{j+1}\Bigr\}.
\]
Then $0<\beta_j\leq h(\epsilon_{j+1})\leq h(\epsilon_j)$ and $\beta_j\to h(0)$ as $j\to\infty$. For each $j$ let $f_j:=f_{h,\epsilon_j}$ be the mountain given by Lemma \ref{mountain}, defined on $[0,\ell_j]$ with $\ell_j:=\log(1/\epsilon_j)$, and let $x_j\in[0,\ell_j-\log(1/\theta_j)]$ be the point given by (M3) for $\theta=\theta_j\in[\epsilon_j,1]$.

For $j\geq1$ and $F\geq0$ we also define a \emph{valley} $v_{j,F}$ on $[0,T_j+F]$, where $T_j:=\Delta_j+\Delta'_j$ with
\[
\Delta_j:=\log\frac{h(\epsilon_j)}{\beta_j}\geq0\qquad\text{and}\qquad \Delta'_j:=\log\frac{1-\beta_j}{1-h(\epsilon_{j+1})}\geq0,
\]
by
\[
v_{j,F}(t):=
\begin{cases}
h(\epsilon_j)\,e^{-t} & \text{if } t\in[0,\Delta_j],\\
\beta_j & \text{if } t\in[\Delta_j,\Delta_j+F],\\
1-(1-\beta_j)\,e^{-(t-\Delta_j-F)} & \text{if } t\in[\Delta_j+F,\,T_j+F].
\end{cases}
\]
The function $v_{j,F}$ is continuous, $v_{j,F}(0)=h(\epsilon_j)$, $v_{j,F}(T_j+F)=h(\epsilon_{j+1})$, $\beta_j\leq v_{j,F}\leq h(\epsilon_j)$, and $D_-v_{j,F}(t)\in[-v_{j,F}(t),1-v_{j,F}(t)]$ for every $t\in(0,T_j+F]$.

\emph{Step 3. The function $G$.} Let $w_0$ be the constant given by Lemma \ref{BR lemma 3.4} for functions $\tilde g$ with $\tilde g(0)=h(\epsilon_1)\in(0,1)$. By \eqref{U window} we can choose inductively numbers $X_1<X_2<\cdots$ as follows. Let $X_1\geq w_0+x_1$ be such that
\[
\inf_{y\in[X_1,X_1+\log(1/\theta_1)]}S_1(y)\geq U(\theta_1)-1,
\]
and put $P_1:=X_1-x_1\geq w_0$ and $R_1:=P_1+\ell_1$. Once $P_j$ and $R_j=P_j+\ell_j$ have been defined, choose $X_{j+1}\geq R_j+T_j+x_{j+1}+1$ such that
\begin{equation}\label{good window}
\inf_{y\in[X_{j+1},X_{j+1}+\log(1/\theta_{j+1})]}S_1(y)\geq U(\theta_{j+1})-\frac{1}{j+1},
\end{equation}
and set $P_{j+1}:=X_{j+1}-x_{j+1}$, $F_j:=P_{j+1}-R_j-T_j\geq1$ and $R_{j+1}:=P_{j+1}+\ell_{j+1}$. Since $x_j\leq\ell_j$ we have $R_j\geq X_j$, and therefore $X_{j+1}\geq X_j+1$. In particular $P_1<R_1<P_2<R_2<\cdots$ and $P_j\to\infty$. Define $G:\mathbb{R}\to\mathbb{R}$ by
\[
G(x):=
\begin{cases}
h(\epsilon_1) & \text{if } x\leq P_1,\\
f_j(x-P_j) & \text{if } x\in[P_j,R_j],\ j\geq1,\\
v_{j,F_j}(x-R_j) & \text{if } x\in[R_j,P_{j+1}],\ j\geq1.
\end{cases}
\]
By (M1) and the properties of the valleys, $G$ is well defined and continuous, and $G$ takes values in $(0,1)$, since $0<\beta_j\leq G\leq B<1$ on $[R_j,P_{j+1}]$ and $h(\epsilon_j)\leq G\leq B$ on $[P_j,R_j]$. Moreover, by (M2) and the corresponding property of the valleys, $D_-G(x)\in[-G(x),1-G(x)]$ for every $x\in\mathbb{R}$, because the left Dini derivative at each point only depends on the piece of $G$ lying to its left. By Lemma \ref{BR lemma 2.4} with $V=\mathbb{R}$, we obtain $G\in\mathcal{G}(0,1)$.

We now collect three properties of $G$.

\begin{enumerate}
\item[(G1)] For every $j\geq1$, $G(y)\geq h(\theta_j)$ for all $y\in[X_j,X_j+\log(1/\theta_j)]$.
\item[(G2)] For every $\theta\in(0,1]$, $\displaystyle\limsup_{x\to\infty}\ \inf_{y\in[x,x+\log(1/\theta)]}G(y)\leq h(\theta)$.
\item[(G3)] $\displaystyle\liminf_{x\to\infty}G(x)=h(0)$.
\end{enumerate}

Property (G1) follows from (M3), since $[X_j,X_j+\log(1/\theta_j)]=P_j+[x_j,x_j+\log(1/\theta_j)]\subset[P_j,R_j]$. For (G2), let $\theta\in(0,1]$, let $x\geq P_1$ and let $W=[x,x+\log(1/\theta)]$. If $W\subset[P_j,R_j]$ for some $j$, then necessarily $\theta\geq\epsilon_j$, and (M4) gives $\inf_W G\leq h(\theta)$. Otherwise, since $W$ is an interval contained in $[P_1,\infty)=\bigcup_{j}\bigl([P_j,R_j]\cup[R_j,P_{j+1}]\bigr)$, it intersects some valley interval $[R_j,P_{j+1}]$ with $P_{j+1}\geq x$, on which $G\leq h(\epsilon_j)$. Let $j(x):=\min\{j : P_{j+1}\geq x\}$, so that $j(x)\to\infty$ as $x\to\infty$. Since $h(\epsilon_j)$ is non-increasing in $j$, in all cases
\[
\inf_{y\in W}G(y)\leq\max\{h(\theta),\,h(\epsilon_{j(x)})\}.
\]
As $j\to\infty$, $h(\epsilon_j)$ decreases to $\inf_{t\in(0,1]}h(t)\leq h(\theta)$, and (G2) follows. For (G3), note that on $[P_1,\infty)$ we have $G\geq\min\{h(\epsilon_j),\beta_j : j\geq1\}\geq h(0)$, while $G(R_j+\Delta_j)=\beta_j\to h(0)$ and $R_j\to\infty$.

Let $\tilde g(y):=G(y+w_0)$. Then $\tilde g\in\mathcal{G}(0,1)$, $\tilde g(0)=G(w_0)=h(\epsilon_1)$ because $w_0\leq P_1$, and $\tilde g$ is not rapidly decreasing, since by (M1) we have $\max_{[P_j,R_j]}G=B$ for every $j$, so that $\limsup_{y\to\infty}\tilde g(y)=B>0$. Moreover $\kappa_{w_0}(\tilde g)(x)=G(x)$ for all $x\geq w_0$. By Lemma \ref{BR lemma 3.4} (with $w=w_0$) there is a sequence $\mathbf{r}=(r_j)\subset(0,1/2]$ such that, writing $\rho^{(2)}_0:=1$, $\rho^{(2)}_k:=r_1\cdots r_k$ and $S_2:=S_{\rho^{(2)}}$, we have
\begin{equation}\label{approx G}
|S_2(x)-G(x)|\leq\log2\,\exp(-x)\qquad\text{for all }x\geq w_0 .
\end{equation}

\emph{Step 4. The combined scales.} Define
\[
\rho_k:=\min\{\rho^{(1)}_k,\rho^{(2)}_k\}\qquad(k\geq0),
\]
so that $\rho_0=C^{-1}<1$. If $\rho_{k-1}=\rho^{(i)}_{k-1}$ with $i\in\{1,2\}$, then $\rho_k\leq\rho^{(i)}_k\leq\rho^{(i)}_{k-1}/2=\rho_{k-1}/2$, because $s_k(a)/s_{k-1}(a)\leq1/\kappa_1<1/2$ and $r_k\leq1/2$. Hence
\begin{equation}\label{ratio half}
\rho_k\leq\rho_{k-1}/2\qquad\text{for all }k\geq1,
\end{equation}
and in particular $\rho$ is strictly decreasing and converges to $0$. Since both $\rho^{(1)}$ and $\rho^{(2)}$ are decreasing, for $\phi\in(0,\rho_0)$ we have $\rho_k\leq\phi$ if and only if $\rho^{(1)}_k\leq\phi$ or $\rho^{(2)}_k\leq\phi$, and therefore $k_\rho(\phi)=\min\{k_{\rho^{(1)}}(\phi),k_{\rho^{(2)}}(\phi)\}$. By \eqref{def s rho},
\begin{equation}\label{min identity}
s_\rho=\min\{s_{\rho^{(1)}},s_{\rho^{(2)}}\}\ \text{ on }(0,\rho_0),
\end{equation}
that is, $S_\rho=\min\{S_1,S_2\}$ for all sufficiently large $x$.

We claim that
\begin{equation}\label{claim window}
\begin{split}
\Phi(\theta):=\limsup_{x\to\infty}\ \inf_{y\in[x,x+\log(1/\theta)]}S_\rho(y)&=h(\theta)\qquad\text{for all }\theta\in(0,1],\\
\liminf_{x\to\infty}S_\rho(x)&=h(0).
\end{split}
\end{equation}
Since $S_\rho\leq S_2\leq G+\log2\,e^{-x}$ by \eqref{min identity} and \eqref{approx G}, property (G2) gives $\Phi(\theta)\leq h(\theta)$. For the reverse inequality, fix $q\in\mathbb{Q}\cap(0,1)$ and let $j$ be any of the infinitely many indices with $\theta_j=q$. On the window $[X_j,X_j+\log(1/q)]$ we have $S_1\geq U(q)-1/j\geq h(q)-1/j$ by \eqref{good window} and the hypothesis $h\leq U$, and $S_2\geq G-\log2\,e^{-X_j}\geq h(q)-\log2\,e^{-X_j}$ by (G1) and \eqref{approx G}. Hence
\[
\inf_{y\in[X_j,X_j+\log(1/q)]}S_\rho(y)\geq h(q)-\frac1j-\log2\,e^{-X_j},
\]
and since $X_j\to\infty$ along these indices, $\Phi(q)\geq h(q)$. Now $\Phi$ is non-decreasing in $\theta$, because the windows become shorter as $\theta$ increases. Thus, if $\theta\in(0,1]$ and $q_n\in\mathbb{Q}\cap(0,1)$ increase to $\theta$, then $\Phi(\theta)\geq\Phi(q_n)\geq h(q_n)\to h(\theta)$ by the continuity of $h$ on $(0,1]$. This proves the first part of \eqref{claim window}. For the second part, \eqref{min identity}, \eqref{U window}, (G3) and \eqref{approx G} give
\[
\liminf_{x\to\infty}S_\rho(x)=\min\Bigl\{\liminf_{x\to\infty}S_1(x),\ \liminf_{x\to\infty}S_2(x)\Bigr\}=\min\{U(0),h(0)\}=h(0),
\]
where we used the hypothesis $h(0)\leq U(0)$.

\emph{Step 5. A complementary Cantor set with scales $\rho$.} For $m\geq0$ let
\[
\pi_m:=\prod_{i=1}^{m}\Bigl(1-\frac{1}{(i+1)^2}\Bigr)=\frac{m+2}{2(m+1)}\in\Bigl[\frac12,1\Bigr]\qquad\text{and}\qquad\rho'_m:=\pi_m\rho_m .
\]
By \eqref{ratio half}, $\rho'_m/\rho'_{m-1}\leq\bigl(1-(m+1)^{-2}\bigr)/2<1/2$, so the numbers
\[
\eta_m:=\rho'_{m-1}-2\rho'_m\qquad(m\geq1)
\]
satisfy $0<\eta_m<\rho'_{m-1}$. Moreover $2^m\rho'_m\leq2^m\rho^{(1)}_m=C^{-1}x_{2^m}(a)\to0$, so a telescoping argument gives
\begin{equation}\label{telescoping}
\sum_{m>k}2^{m-k-1}\eta_m=\rho'_k\qquad\text{for all }k\geq0 .
\end{equation}

We now select the gaps. For $p\geq0$ let $\Lambda_p:=\{n\in\mathbb{N} : 2^p\leq n<2^{p+1}\}$, and for $m\geq1$ let $p(m):=\min\{p\geq0 : a_{2^p}\leq\eta_m\}$. By Lemma \ref{consequences H}(a) and the choice of $C$,
\[
\begin{split}
\eta_m<\rho'_{m-1}\leq\rho^{(1)}_{m-1}&=C^{-1}s_{m-1}(a)\leq C^{-1}\kappa_2\,s_m(a)\\
&\leq C^{-1}\kappa_2\frac{\kappa_1}{\kappa_1-2}\,a_{2^m}=a_{2^m}.
\end{split}
\]
Since $a$ is non-increasing, $a_{2^p}\geq a_{2^m}>\eta_m$ for $p\leq m$, hence $p(m)\geq m+1$. If $n\in\Lambda_{p(m)}$, then $a_n\leq a_{2^{p(m)}}\leq\eta_m$, and by Lemma \ref{consequences H}(b) and the minimality of $p(m)$,
\[
a_n\geq a_{2^{p(m)+1}}\geq c_a^{-2}\,a_{2^{p(m)-1}}>c_a^{-2}\,\eta_m .
\]
Thus
\begin{equation}\label{gap sizes}
c_a^{-2}\,\eta_m<a_n\leq\eta_m\qquad\text{for all }n\in\Lambda_{p(m)} .
\end{equation}
For each $p\geq1$, the levels $m$ with $p(m)=p$ satisfy $m\leq p-1$, so
\[
\sum_{m : p(m)=p}2^{m-1}\leq\sum_{m=1}^{p-1}2^{m-1}<2^{p-1}=\frac{\#\Lambda_p}{2}.
\]
Therefore we can choose pairwise disjoint sets $A_m\subset\Lambda_{p(m)}$, $m\geq1$, with $\#A_m=2^{m-1}$, in such a way that at least $2^{p-1}$ indices of each $\Lambda_p$, $p\geq1$, do not belong to $A:=\bigcup_mA_m$. For each $m\geq1$ fix a bijection between $\{1,2\}^{m-1}$ and $A_m$, and for $\sigma\in\{1,2\}^{m-1}$ let $b_\sigma:=a_n$, where $n\in A_m$ corresponds to $\sigma$.

For $\sigma\in\Sigma^*$ let $\ell(\sigma):=\sum_{\tau\succeq\sigma}b_\tau$. If $|\sigma|=k$, then for each $m>k$ there are exactly $2^{m-1-k}$ words $\tau\succeq\sigma$ with $|\tau|=m-1$, and by \eqref{gap sizes}, \eqref{telescoping} and $\pi_k\geq1/2$,
\begin{equation}\label{lengths scheme}
\frac{\rho_k}{2c_a^2}\leq\frac{\rho'_k}{c_a^{2}}\leq\ell(\sigma)\leq\rho'_k\leq\rho_k .
\end{equation}
Note also that $\ell(\sigma)=\ell(\sigma1)+b_\sigma+\ell(\sigma2)$. We define intervals recursively by $I_\varnothing:=[0,\ell(\varnothing)]$ and, if $I_\sigma=[u,u+\ell(\sigma)]$,
\[
I_{\sigma1}:=[u,\,u+\ell(\sigma1)]\qquad\text{and}\qquad I_{\sigma2}:=[u+\ell(\sigma1)+b_\sigma,\,u+\ell(\sigma)],
\]
and we let $J_\sigma:=\bigl(u+\ell(\sigma1),\,u+\ell(\sigma1)+b_\sigma\bigr)$ be the open interval between them.
By \eqref{lengths scheme}, $\{I_\sigma\}$ is a $(\rho,2c_a^2)$-homogeneous binary Cantor scheme. Let $M$ be its limit set. A point of $I_\varnothing$ which does not belong to $M$ lies in some $I_\sigma$ but in none of its two children, so it lies in $J_\sigma$. Conversely $J_\sigma\cap M=\varnothing$. Hence
\[
I_\varnothing\setminus M=\bigcup_{\sigma\in\Sigma^*}J_\sigma,
\]
a disjoint union of open intervals whose lengths are the numbers $a_n$, $n\in A$, each of them used exactly once. Furthermore $\ell(\varnothing)=\sum_{n\in A}a_n<1$. By Lemma \ref{dimension scheme} and \eqref{claim window},
\begin{equation}\label{dim M}
\uda M=h(\theta)\ \text{ for all }\theta\in(0,1]\qquad\text{and}\qquad\dimh M=h(0).
\end{equation}

\emph{Step 6. The countable part and conclusion.} Let $a'=\{a'_i\}$ be the non-increasing rearrangement of $\{a_n : n\notin A\}$, which is an infinite sequence with $\sum_ia'_i=1-\ell(\varnothing)$. Define
\[
y_k:=\ell(\varnothing)+\sum_{i\geq k}a'_i\quad(k\geq1),\qquad D':=\{y_k : k\geq1\}\cup\{\ell(\varnothing)\},
\]
and $E:=M\cup D'$.
Then $y_1=1$, $y_k$ decreases to $\ell(\varnothing)$, and $E$ is a closed subset of $[0,1]$ such that
\[
[0,1]\setminus E=\bigcup_{\sigma\in\Sigma^*}J_\sigma\ \cup\ \bigcup_{k\geq1}(y_{k+1},y_k),
\]
which is a disjoint union of open intervals whose lengths are exactly the terms of $a$. Thus $E\in\mathcal{C}_a$. The set $D'$ is a translate of $D_{a'}$, so by Lemma \ref{Maxmin}, $\uda D'=\varphi_{B'}(\theta)$ with $B':=\ubd D'\leq\ubd E=B$. Since $\varphi_B(\theta)$ is non-decreasing in $B$, \eqref{h positive} gives $\uda D'\leq\varphi_B(\theta)\leq h(\theta)$. By the finite stability of the upper intermediate dimensions (see \cite{FFK}) and \eqref{dim M}, for every $\theta\in(0,1]$,
\[
\uda E=\max\{\uda M,\ \uda D'\}=\max\{h(\theta),\varphi_{B'}(\theta)\}=h(\theta).
\]
Finally, for $\theta=0$, $\dimh E=\max\{\dimh M,\dimh D'\}=\max\{h(0),0\}=h(0)$. This completes the proof.
\end{proof}

As a consequence, we obtain a complete description of the family $\mathcal{F}^U(a)=\{\theta\mapsto\uda E : E\in\mathcal{C}_a\}$.

\begin{cor}\label{characterization}
Let $a=\{a_n\}$ be a positive, non-increasing sequence with $\sum_na_n=1$ satisfying \eqref{hyp H}. Then $\mathcal{F}^U(a)$ is exactly the set of functions $h\in\mathcal{H}(0,1)$ such that
\[
h(1)=\ubd C_a\qquad\text{and}\qquad h(\theta)\leq\uda C_a\quad\text{for all }\theta\in[0,1].
\]
Moreover, every $h\in\mathcal{F}^U(a)$ satisfies $\uda D_a\leq h(\theta)\leq\uda C_a$ for all $\theta\in[0,1]$, and if $h(\theta_0)=\overline{\dim}_{\,\theta_0}D_a$ for some $\theta_0\in(0,1]$, then $h(\theta)=\uda D_a$ for all $\theta\in[\theta_0,1]$.
\end{cor}

\begin{proof}
The inclusion $\supseteq$ is Theorem \ref{main functions} with $\lambda=0$ and $\alpha=1$. Conversely, let $E\in\mathcal{C}_a$ and $h(\theta)=\uda E$. By Theorem \ref{BanajiRutar} and Remark \ref{inclusion H}, $h\in\mathcal{H}(\dim_L E,\dim_A E)\subseteq\mathcal{H}(0,1)$. Moreover $h(1)=\ubd E=\ubd C_a$. Since \eqref{hyp H} implies $\lim_{n\to\infty}(s_n/s_{n+1})^{1/n}=1$, Lemma \ref{Maxmin} gives $h(\theta)\leq\uda C_a$ for $\theta\in(0,1]$, and for $\theta=0$ this inequality is $\dimh E\leq\dimh C_a$, which holds by Proposition \ref{hausdorff cantor}. The last statement follows from Lemma \ref{consequences H}(c) and Proposition \ref{lower bound rigidity}, since $\uda D_a=\varphi_B(\theta)$ with $B=\ubd C_a\in(0,1)$.
\end{proof}

In particular, the functions $h_s$, $s>1$, of Remark \ref{rmk rigidity} form only a small part of $\mathcal{F}^U(a)$. We conclude with two remarks on the construction.

\begin{rmk}\label{rmk minimum}
The minimum taken in Step 4 of the proof of Theorem \ref{main functions} is useful because, in general, the hypothesis $h\leq\uda C_a$ does not provide the pointwise lower bounds on $S_1$ needed in the construction. Indeed, let $b$ be the non-increasing rearrangement of a subfamily of the terms of $a$. Then $b_n\leq a_n$ for every $n$, hence $s_k(b)\leq s_k(a)$ for every $k$, and therefore $s_{\rho(b)}\leq s_{\rho(a)}$ pointwise, where $\rho(b)_k=s_k(b)$ and $\rho(a)_k=s_k(a)$. Thus, the function governing the covering numbers of a decreasing Cantor set $C_b$ built from such a subfamily is dominated, at every scale, by the corresponding function for $C_a$. By contrast, the hypothesis $h\leq\uda C_a$ only gives a comparison of the limsup of window infima. The minimum in Step 4 allows the construction to retain the required lower bounds from $S_1$ on the selected windows while also enforcing the desired upper bounds through the second scale function.
\end{rmk}

\begin{rmk}\label{rmk why schemes}
The decreasing Cantor sets used in the proof of Theorem \ref{attainablevaluesangelini} are not sufficient to realize every function in $\mathcal{F}^U(a)$. Let $b$ be a non-increasing summable sequence with
\[
2<\inf_n \frac{s_n(b)}{s_{n+1}(b)}\leq\sup_n\frac{s_n(b)}{s_{n+1}(b)}<\infty,
\]
which is the property of the sequence $b$ used in that proof. Arguing as in Lemma \ref{consequences H}(c) we get $0<\dim_L C_b\leq\dim_A C_b<1$. By Theorem \ref{BanajiRutar} and Remark \ref{inclusion H}, the function $g(\theta)=\uda C_b$ is either constant or satisfies $D^+g(\theta)<g(\theta)(1-g(\theta))/\theta$ at every $\theta\in(0,1)$.

Now let $a_{2^k+i}=3^{-(k+1)}$ for $k\geq0$ and $0\leq i<2^k$. Then $\sum_na_n=1$, $s_n(a)=3^{-n}$, so $a$ satisfies \eqref{hyp H}, and $C_a$ is the middle third Cantor set. Since its Hausdorff and box dimensions coincide, $\uda C_a=B=\log2/\log3$ for all $\theta\in[0,1]$. Fix $K>B/(1-B)$, let $\vartheta_2:=B/(K(1-B))\in(0,1)$, choose $\vartheta_1\in(0,\vartheta_2)$ and define
\[
h(\theta):=
\begin{cases}
K\vartheta_1/(1+K\vartheta_1) & \text{if }\theta\in[0,\vartheta_1],\\
K\theta/(1+K\theta) & \text{if }\theta\in[\vartheta_1,\vartheta_2],\\
B & \text{if }\theta\in[\vartheta_2,1].
\end{cases}
\]
This function is continuous and non-decreasing, $h(1)=B$, $h\leq\uda C_a$, and $D^+h(\theta)=h(\theta)(1-h(\theta))/\theta$ for $\theta\in[\vartheta_1,\vartheta_2)$, while $D^+h=0$ elsewhere in $(0,1)$. Hence $h\in\mathcal{F}^U(a)$ by Corollary \ref{characterization}. Moreover $Q_h>B/(1-B)$ on $(0,1)$, so $h>\varphi_B$ on $(0,1)$. Suppose that $h(\theta)=\overline{\dim}_{\,\theta}(C_b\cup D')$ for a set $C_b\cup D'\in\mathcal{C}_a$ as in the proof of Theorem \ref{attainablevaluesangelini}, where $b$ is as above and $D'$ is a translate of $D_{a'}$ for some non-increasing sequence $a'$. By finite stability and Lemma \ref{Maxmin}, $h=\max\{\uda C_b,\varphi_{B'}\}$ with $B'=\ubd D'\leq B$. Since $h>\varphi_B\geq\varphi_{B'}$ on $(\vartheta_1,\vartheta_2)$, we get $\uda C_b=h(\theta)$ on this open interval. But $h$ is strictly increasing and satisfies $D^+h(\theta)=h(\theta)(1-h(\theta))/\theta$ on $(\vartheta_1,\vartheta_2)$, which is impossible by the previous paragraph. The sets $M$ constructed in Step 5 avoid this obstruction, since the ratios $\rho_{k}/\rho_{k-1}$ may approach $1/2$ or $0$, and the gaps are not required to be arranged in decreasing order.
\end{rmk}

\section*{Further Work}

In \cite{Banaji-R22}, the authors characterized the functions which arise as the intermediate dimensions of bounded subsets of $\mathbb{R}^n$ (see Theorem \ref{BanajiRutar}). Theorem \ref{main functions} and Corollary \ref{characterization} provide the analogous characterization for the upper intermediate dimensions of complementary sets, under hypothesis \eqref{hyp H}. The corresponding problem for the lower intermediate dimensions remains open.

Let us briefly describe what is known in that case. Since $\da C_a=\lbd C_a$ for all $\theta$, every function $h(\theta)=\da E$ with $E\in\mathcal{C}_a$ satisfies $h(\theta)\leq h(1)=\lbd C_a$. By Theorem \ref{BanajiRutar} and Remark \ref{inclusion H}, $h\in\mathcal{H}(0,1)$, and if $0<\lbd C_a<1$, Proposition \ref{lower bound rigidity} applies to $h$ with $B$ replaced by $\lbd C_a$, so that $\da D_a\leq h(\theta)$ with the same rigidity property. The main difficulty is that the lower intermediate dimensions are not finitely stable, so the gluing argument of Step 6 in the proof of Theorem \ref{main functions} does not control the lower intermediate dimensions of the union. The following questions arise naturally.

\begin{question}
Let $a=\{a_n\}$ be a non-increasing summable sequence satisfying
\[
2<\inf_n\frac{s_n(a)}{s_{n+1}(a)}\leq \sup_n\frac{s_n(a)}{s_{n+1}(a)}<\infty .
\]
Is it true that
\[
\mathcal{F}^L(a):=\bigl\{\,\theta\mapsto \da E : E\in\mathcal{C}_a\,\bigr\}=\bigl\{h\in\mathcal{H}(0,1) : h(1)=\lbd C_a\bigr\}?
\]
More generally, which pairs of functions $(h_1,h_2)$ arise as $(\da E,\uda E)$ for some $E\in\mathcal{C}_a$?
\end{question}

\addtocontents{toc}{\protect\setcounter{tocdepth}{0}}
\section*{Funding}
This work was supported by FIS-2023-02725, singular structures in the
geometry of measures: decompositions, rigidity and rectifiability, CUP
E53C25001800001. The author was additionally partially supported by
grants PICT 2022-4875 (ANPCyT), PIP 202287/22 (CONICET), and PROICO 3-0720
``An\'alisis Real y Funcional. Ec. Diferenciales''.
\addtocontents{toc}{\protect\setcounter{tocdepth}{2}}

\end{document}